\RequirePackage{fix-cm}
\documentclass[smallextended]{svjour3}       % onecolumn (second format)

\smartqed

\usepackage{graphicx,mathptmx,latexsym,amssymb,amsmath,marvosym}

\usepackage{color}

\newcommand{\envelope}{(\raisebox{-.5pt}{\scalebox{1.45}{\Letter}}\kern-1.7pt)}

\journalname{Numerische Mathematik}

\begin{document}

\title{A uniformly convergent difference scheme on a modified Shishkin mesh for the singular perturbation boundary value problem}

\titlerunning{A uniformly convergent difference scheme on a modified Shishkin mesh}

\author{Enes Duvnjakovi\'c \and Samir Karasulji\'c \and {\\} Vedad Pasic \and
   Helena Zarin}

\institute{E. Duvnjakovi\' c  \and S Karasulji\' c  \and V. Pasic \envelope \at Department of Mathematics, University of Tuzla, Univerzitetska 4, 75000 Tuzla, Bosnia and Herzegovina \\
\email{vedad.pasic@untz.ba} {\\}  ~ {\\}
E. Duvnjakovi\' c \\  \email{enes.duvnjakovic@untz.ba} {\\} ~ {\\}
S. Karasulji\' c \\  \email{samir.karasuljic@untz.ba}  {\\}  ~ {\\}
H. Zarin \at Department of Mathematics and Informatics, University of Novi Sad, Trg Dositeja Obradovi\' ca, 21000 Novi Sad, Serbia \\
\email{helena.zarin@dmi.uns.ac.rs}
}
\date{Received: date / Accepted: date}

\maketitle

%%%%% Begin Abstract %%%%%%%%%%%
\begin{abstract}
In this paper we are considering a semilinear singular perturbation reaction -- diffusion boundary value problem, which contains a small perturbation parameter that acts on the highest order derivative. We construct a difference scheme on an arbitrary nonequidistant mesh using a collocation method and Green's function. We show that the constructed difference scheme has a unique solution and that the scheme is stable. The central result of the paper is $\epsilon$-uniform convergence of almost second order for the discrete approximate solution on a modified Shishkin mesh. We finally provide two numerical examples which illustrate the theoretical results on the uniform accuracy of the discrete problem, as well as the robustness of the method.
\end{abstract}
%%%%% end %%%%%%%%%%%

\section{Introduction}\label{sec1}
We consider the semilinear singularly perturbed problem
\begin{align}
\epsilon^2y''(x)&=f(x,y)\:\quad \textrm{ on }\:\left[ 0,1\right],
\label{uvod1}
\\
y(0)&=0,\ y(1)=0,
\label{uvod2}
\end{align}
where $0<\epsilon<1.$ We assume that the nonlinear function $f$ is continuously differentiable,
i.e. that $ \displaystyle f\in C^{k}\left( [0,1]\times \mathbb{R}\right)$, for  $k\geqslant 2$ and that $f$ has a strictly positive derivative with respect to $y$
\begin{equation}
\dfrac{\partial f}{\partial y}=f_y\geqslant m>0\:\text{ on }\:\left[0,1\right]\times \mathbb{R}\:\: \quad (m=const).
\label{uvod3}
\end{equation}

The solution $y$ of the problem \eqref{uvod1}--\eqref{uvod3} exhibits sharp boundary layers at the endpoints of $\left[0,1\right]$ of $\cal{O} ( \epsilon\ln$ $1/\epsilon)$ width. It is well known that the standard discretization methods for solving \eqref{uvod1}--\eqref{uvod3} are unstable and do not give accurate results when the perturbation parameter $\epsilon$ is smaller than some critical value. With this in mind, we therefore need to develop a method which produces a numerical solution for the starting problem with a satisfactory value of the error. Moreover, we additionally require that the error does not depend on $\epsilon$; in this case we say that the method is uniformly convergent with respect to $\epsilon$ or $\epsilon$-uniformly convergent.

Numerical solutions $\overline{y}$ of given continuous problems obtained using a $\epsilon$-uniformly convergent method satisfy the condition
\[\left\|y-\overline{y}\right\|\leqslant C\kappa(N),\quad\kappa(N)\rightarrow 0,\quad N\rightarrow +\infty,\]
where $y$ is the exact solution of the original continuous problem, $\left\|\cdot\right\|\:$ is the discrete maximum norm, $N$ is the number of mesh points that is independent of $\epsilon$ and $C>0$ is a constant which does not depend on $N$ or $\epsilon$. We therefore demand that the numerical solution $\overline{y}$ converges to $y$ for every value of the perturbation parameter in the domain $0<\epsilon<1$ with respect to the discrete maximum norm $\left\|\cdot\right\|.$

The problem \eqref{uvod1}--\eqref{uvod2} has been researched by many authors with various assumptions on $f(x,y)$. %The greatest number of results in numerically solving the original problem \eqref{uvod1}--\eqref{uvod2} have been obtained by using the finite difference method and its modifications, as well as collocation methods with polynomial splines. %Some of the first results of this problem show uniform convergence of constructed difference schemes, so the next logical step was constructing schemes on suitable meshes which have $\epsilon$-uniform convergence.
Various different difference schemes have been constructed which are uniformly convergent on equidistant meshes as well as schemes on specially constructed, mostly Shishkin and Bakvhvalov-type meshes, where $\epsilon$-uniform convergence of second order has been demonstrated, see e.g. \cite{D4,HSR,SS,V1,SU1}, as well as schemes  with $\epsilon$-uniform convergence of order greater than two, see e.g. \cite{H1,HH,H2,V2,V3}. These difference schemes were usually constructed using the finite difference method and its modifications or collocation methods with polynomial splines. A large number of difference schemes also belongs to the group of exponentially fitted schemes or their uniformly convergent versions. Such schemes were mostly used in numerical solving of corresponding linear singularly perturbed boundary value problems on equidistant meshes, see e.g. \cite{S1,OS,R,S9}. Less frequently were used for numerical solving of nonlinear singularly perturbed boundary value problems, see e.g. \cite{KN,ORS1}.
%The common thread in all uniformly convergent exponentially fitted schemes is the fact that the scheme coefficients contain exponential or hyperbolic functions which do not act in the same manner as the solution of the original boundary value problem.

Our present work represents a synthesis of these two approaches, i.e. we want to construct a difference scheme which belongs to the group of exponentially fitted schemes and apply this scheme to a corresponding nonequidistant layer-adapted mesh. The main motivation for constructing such a scheme is obtaining an $\epsilon$-uniform convergent method, which will be guaranteed by the layer-adapted mesh, and then further improving the numerical results by using an exponentially fitted scheme. We therefore aim to construct an $\epsilon$-uniformly convergent difference scheme on a modified Shishkin mesh, using the results on solving linear boundary value problems obtained by  Roos \cite{R}, O'Riordan and Stynes \cite{OS} and Green's function for a suitable operator.

This paper has the following structure. Section \ref{sec1}. provides background information and introduces the main concepts used throughout. In Section \ref{sec2}. we construct our difference scheme based on which we generate the system of equations whose solving gives us the numerical solution values at the mesh points. We also prove the existence and uniqueness theorem for the numerical solution. In Section \ref{sec3}. we construct the mesh, where we use a modified Shiskin mesh with a smooth enough generating function in order to discretize the initial problem. In Section \ref{sec4}. we show $\epsilon$-uniform convergence and its rate. In Section \ref{sec5}. we provide some numerical experiments and discuss our results and possible future research.

{\bf Notation.} \hspace{0.04cm} Throughout this paper we denote by $C$ (sometimes subscripted) a generic positive constant that may take different values in different formulae, always independent of $N$ and $\epsilon$. We also (realistically) assume that $\epsilon\leqslant \frac{C}{N}$. Throughout the paper, we denote by $\left\|\cdot\right\|$  the usual discrete maximum norm
$\displaystyle
\left\|u\right\|=\displaystyle\max_{0\leqslant i\leqslant N}\left|u_i\right|,\:u\in\mathbb{R}^{N+1},
$ as well as the corresponding matrix norm.
\section{Scheme construction} \label{sec2}
Consider the differential equation  (\ref{uvod1}) in an equivalent form
\begin{equation*}
L_{\epsilon}y(x):=\epsilon^2y''(x)- \gamma y(x)=\psi(x,y(x))\quad\text{on}\quad\left[0,1\right],
\label{konst1}
\end{equation*}
where
\begin{equation}
\psi(x,y)=f(x,y)-\gamma y,
\label{konst2}
\end{equation}
and $\gamma\geqslant m$ is a chosen constant.
In order to obtain a difference scheme needed to calculate the numerical solution of the boundary value problem \eqref{uvod1}--\eqref{uvod2}, using an arbitrary mesh $0=x_0<x_1<x_2<\ldots<x_N=1$ we construct a solution of the following boundary value problem
\begin{align}
\label{konst31}
L_{\epsilon}y_i(x)&=\psi(x,y_i(x))\quad\text{ on }\:\left( x_{i},x_{i+1}\right),\\
y_i(x_i)&=y(x_i),\quad
y_i(x_{i+1})=y(x_{i+1}),
\label{konst32}
\end{align}
for $i=0,1,\ldots,N-1.$
It is clear that $y_i(x)\equiv y(x)\text{ on }\left[x_i,x_{i+1} \right],\ i=0,1,\ldots,N-1.$ The solutions of corresponding homogenous boundary value problems
\begin{align*}
   L_{\epsilon}
        u_i^{I}(x)&:=0\quad\text{on}\quad\left(x_{i},x_{i+1}\right), &  L_{\epsilon}
        u_i^{II}(x)&:=0\quad \text{on}\quad\left(x_{i},x_{i+1}\right),\\
	    u_i^{I}(x_i)&=1,\quad u_i^{I}(x_{i+1})=0, &  u_i^{II}(x_i)&=0,\quad u_i^{II}(x_{i+1})=1,
\end{align*} for $i=0,1,\ldots,N-1$,
are known, see \cite{R}, i.e.
\begin{equation*}
    u_{i}^{I}(x)=\dfrac{\sinh\left(\beta\left(x_{i+1}-x\right)\right)}{\sinh\left(\beta h_i \right)}
     \quad \text{ and }\quad
    u_{i}^{II}(x)=  \dfrac{\sinh\left(\beta\left(x-x_i\right) \right)}{\sinh\left(\beta h_i\right)},
\end{equation*}
for $i=0,1,\ldots,N-1$, where $\displaystyle x\in\left[x_{i},x_{i+1}\right],\:\beta=\frac{\sqrt{\gamma}}{\epsilon},\:h_i=x_{i+1}-x_i.$
The solution of \eqref{konst31}--\eqref{konst32} is given by \[y_i(x)=C_1u_{i}^{I}(x)+C_2u_{i}^{II}(x)+\int_{x_{i}}^{x_{i+1}}{G_i(x,s)\psi(s,y(s))ds},\quad x\in\left[x_{i},x_{i+1}\right],\]
where $G_i(x,s)$ is the Green's function associated with the operator $L_{\epsilon}$ on the interval $\left[x_{i},x_{i+1} \right]$. The function $G_i(x,s)$ in this case has the following form
\begin{equation*}
G_i(x,s)=\dfrac{1}{\epsilon^2w_i(s)}\left\{\begin{array}{c}u_i^{II}(x)u_i^{I}(s),\quad x_i\leqslant x\leqslant s\leqslant x_{i+1},\\  u_i^{I}(x)u_i^{II}(s),\quad x_i\leqslant s\leqslant x\leqslant x_{i+1},	
\end{array} \right.
\end{equation*}
where $w_i(s)=u_i^{II}(s)\left( u_i^{I}\right)'(s)-u_{i}^{I}(s)\left( u_i^{II}\right)'(s).$
Clearly $w_i(s)\neq 0,\ s\in\left[ x_i,x_{i+1}\right]$. It follows from the boundary conditions (\ref{konst32}) that $C_1=y(x_i)=:y_i,$ $C_2=y(x_{i+1})=:y_{i+1},$  $\:i=0,1,\ldots,N-1.$
Hence, the solution $y_i(x)$ of \eqref{konst31}--\eqref{konst32} on $\left[ x_{i},x_{i+1}\right]$ has the following form
\begin{equation}
y_i(x)=y_iu_i^{I}(x)+y_{i+1}u_i^{II}(x)+\int_{x_i}^{x_{i+1}}{G_i(x,s)\psi(s,y(s))ds}.
\label{konst7}
\end{equation}
The boundary value problem
\begin{align*}
L_{\epsilon}y(x):=&\psi(x,y)\qquad\text{on}\quad\left(0,1\right),\\y(0)=& \ y(1)=0,
\end{align*}
has a unique continuously  differentiable solution $y\in C^{k+2}(0,1)$. Since $y_i(x)\equiv y(x)$ on $\left[x_i,x_{i+1} \right]$, $i=0,1,\ldots,N-1$,  we have that
$y'_i(x_i)=y'_{i-1}(x_i)$, for $i=1,2,\ldots,N-1.$
Using this in differentiating (\ref{konst7}), we get that
\begin{align}
\nonumber
y_{i-1}&\left(u_{i-1}^{I}\right)'(x_i)+y_i\left[\left(u_{i-1}^{II}\right)'(x_i) -\left(u_{i}^{I}\right)'(x_i)\right]+y_{i+1}\left[-\left( u_i^{II}\right)'(x_i) \right] & \\
&=\dfrac{\partial}{\partial x}\left[\int_{x_i}^{x_{i+1}}{G_i(x,s)\psi(s,y(s))ds}-\int_{x_{i-1}}^{x_{i}}{G_{i-1}(x,s)\psi(s,y(s))ds} \right]_{x=x_i}. &
\label{konst10}
\end{align}
Since we have that
\begin{eqnarray*}
   \left( u^{I}_{i-i}\right)'(x_i)= \frac{-\beta}{\sinh(\beta h_{i-1})}, \quad  \left( u_i^{II}\right)'(x_i) =\frac{\beta}{\sinh(\beta h_i)},   \\
   \left(u_{i-1}^{II}\right)'(x_i) -\left(u_{i}^{I}\right)'(x_i) =\frac{\beta}{\tanh(\beta h_{i-1})}+\frac{\beta}{\tanh(\beta h_i)} ,
\end{eqnarray*}
equation \eqref{konst10} becomes
\begin{align} \nonumber
\frac{\beta}{\sinh(\beta h_{i-1})}  y_{i-1}&-\left( \frac{\beta}{\tanh(\beta h_{i-1})}+\frac{\beta}{\tanh(\beta h_i)} \right)y_i
+\frac{\beta}{\sinh(\beta h_i)}y_{i+1} & \\
\label{konst13}
      &=\dfrac{1}{\epsilon^2}\left[\int\limits_{x_{i-1}}^{x_{i}}{u_{i-1}^{II}(s)\psi(s,y(s))ds}+\int\limits_{x_{i}}^{x_{i+1}}{u_{i}^{I}(s)\psi(s,y(s))ds}\right],
\end{align}
for $i=1,2,\ldots,N-1$ and $y_0=y_N=0$. We cannot in general explicitly compute the integrals on the RHS of (\ref{konst13}). In order to get a simple enough difference scheme, we approximate the function $\psi$ on $\left[x_{i-1},x_i \right]\cup \left[x_i,x_{i+1} \right]$ using
\begin{equation*}
\overline{\psi}_{i}=\frac{\psi(x_{i-1},\overline{y}_{i-1})+2\psi(x_i,\overline{y}_i)+\psi(x_{i+1},\overline{y}_{i+1})}{4},
\label{konst14}
\end{equation*}
where $\overline{y}_i$ are approximate values of the solution $y$ of the problem (\ref{uvod1})--(\ref{uvod2}) at points $x_i.$
We get that
\begin{flalign*}
&\frac{\beta}{\sinh(\beta h_{i-1})}  \overline{y}_{i-1}
           -\left( \frac{\beta}{\tanh(\beta h_{i-1})}+\frac{\beta}{\tanh(\beta h_i)} \right)\overline{y}_i
               +\frac{\beta}{\sinh(\beta h_i)}\overline{y}_{i+1} &\\
&=\dfrac{1}{\epsilon^2}\frac{\psi(x_{i-1},\overline{y}_{i-1})+2\psi(x_i,\overline{y}_i)+\psi(x_{i+1},\overline{y}_{i+1})}{4} \left[\displaystyle\int\limits_{x_{i-1}}^{x_i}{u_{i-1}^{II}(s)ds}+\displaystyle\int\limits_{x_{i}}^{x_{i+1}}{u_{i}^{I}(s)ds}\right], \\
&=\!\frac{1}{\epsilon^2}\! \frac{\psi(x_{i-1},\overline{y}_{i-1})+2\psi(x_i,\overline{y}_i)+\psi(x_{i+1},\overline{y}_{i+1})}{4} \!\left[\frac{\cosh(\beta h_{i-1})-1}{\beta \sinh(\beta h_{i-1})}
  \! + \! \frac{\cosh(\beta h_{i})-1}{\beta \sinh(\beta h_{i})}\right]\!\!,
\end{flalign*}
for $i=1,2,\ldots,N-1$ and $\overline{y}_0=\overline{y}_N=0$. Using equation \eqref{konst2}, we get that
\begin{align} \nonumber
  (3a_{i}+d_{i}&+\Delta d_{i+1})\left(\overline{y}_{i-1}-\overline{y}_i\right)
     -(3a_{i+1}+d_{i+1}+\Delta d_{i})\left(\overline{y}_i-\overline{y}_{i+1}\right) &\\ \label{konst17}
&-\dfrac{f(x_{i-1},\overline{y}_{i-1})+2f(x_{i},\overline{y}_{i})+f(x_{i+1},\overline{y}_{i+1}) }{\gamma}\left(\Delta d_{i}+\Delta d_{i+1}\right)=0,
\end{align}
for $i=1,2,\ldots,N-1$ and $\overline{y}_0=\overline{y}_N=0$, where
\begin{equation}
 a_i=\frac{1}{\sinh(\beta h_{i-1})},\ d_i=\frac{1}{\tanh(\beta h_{i-1})},\ \Delta d_i=d_i-a_i.
  \label{konst18aa}
\end{equation}
Using the scheme (\ref{konst17}) we form a corresponding discrete analogue of \eqref{uvod1}--\eqref{uvod3}
\begin{flalign}
  F_0\overline{y}&:=\overline{y}_0=0, &  \label{konst18a}\\
  F_i\overline{y}&:=\frac{\gamma}{\Delta d_i+\Delta d_{i+1}}
             \Big[(3a_{i}+d_{i}+\Delta d_{i+1})\left(\overline{y}_{i-1}-\overline{y}_i\right)
               \!- \!(3a_{i+1}+d_{i+1}+\Delta d_{i})\left(\overline{y}_i-\overline{y}_{i+1}\right)& \nonumber \\
                 &\quad \quad -\dfrac{f(x_{i-1},\overline{y}_{i-1})+2f(x_{i},\overline{y}_{i})+f(x_{i+1},\overline{y}_{i+1}) }
                           {\gamma}\left(\Delta d_{i}+\Delta d_{i+1}\right)\Big]=0,& \label{konst18b} \\
  F_N\overline{y}&:=\overline{y}_N=0,
\label{konst18c}
\end{flalign}
where $i=1,2,\ldots,N-1$. The solution $\overline{y}:=\left(\overline{y}_0,\overline{y}_1,\ldots,\overline{y}_N\right)^T$ of the problem \eqref{konst18a}--\eqref{konst18c}, i.e. $F\overline{y}=0,$ where $ F=\left(F_0,F_1,\ldots,F_N\right)^T$ is an approximate solution of the problem \eqref{uvod1}--\eqref{uvod3}.

\begin{theorem} \label{teo21} The discrete problem (\ref{konst18a})--(\ref{konst18c}) has a unique solution $\overline{y}$ for $\gamma\geqslant f_y$. Also, for every $u,v\in\mathbb{R}^{N+1}$ we have the following stabilizing inequality
\begin{equation*}
    \left\|u-v\right\|\leqslant \frac{1}{m}\left\|Fu-Fv\right\|.
\end{equation*}
\end{theorem}
\begin{proof}
We use a technique from \cite{H2} and \cite{V2}, while the proof of existence of the solution of $F\overline{y}=0$ is based on the proof of the relation: $\|\left(F'\right)^{-1}\|\leqslant C,$ where $F'$ is the Fr\'echet derivative of $F$.
The Fr\'echet derivative $H:=F'(\overline{y})$ is a tridiagonal matrix. Let $H=[h_{ij}].$  The non-zero elements of this tridiagonal matrix are
\begin{flalign*}
h_{0,0}=&h_{N,N}=1,\\
h_{i,i}=&\frac{2\gamma}{\Delta d_i+\Delta d_{i+1}}\!\!\left[ -(a_i+a_{i+1})-2(d_{i}+ d_{i+1})
              -\frac{1}{\gamma}\frac{\partial f}{\partial y}(x_{i},\overline{y}_{i})(\Delta d_i+\Delta d_{i+1}) \right]\!\!<0,\\
h_{i,i-1}=&\frac{\gamma}{\Delta d_i+\Delta d_{i+1}}\left[  \left(\Delta d_i+\Delta d_{i+1}\right)\left(1-\frac{1}{\gamma}\frac{\partial f}{\partial y}(x_{i-1},\overline{y}_{i-1})\right)+4 a_i \right]>0,\\
h_{i,i+1}=&\frac{\gamma}{\Delta d_{i}+\Delta d_{i+1}}\left[  \left(\Delta d_{i+1}+\Delta d_{i}\right)\left(1-\frac{1}{\gamma}\frac{\partial f}{\partial y}(x_{i+1},\overline{y}_{i+1})\right)+4a_{i+1} \right]>0,
\end{flalign*}
where $i=1,\ldots,N-1.$
Hence $H$ is an $L-$ matrix. Moreover, $H$ is an $M-$matrix since
\begin{flalign*}
\nonumber
& \left|h_{i,i}\right|-|h_{i,i-1}|-|h_{i-1,i}| = &  \\
&  \!\frac{\gamma}{\Delta d_i+\Delta d_{i+1}} \!\!
    \left[ \frac{(\Delta d_i+\Delta d_{i+1})\!\!\left( \frac{\partial f}{\partial y}(x_{i-1},\overline{y}_{i-1})
       +2\frac{\partial f}{\partial y}(x_{i},\overline{y}_{i})+\frac{\partial f}{\partial y} (x_{i+1},\overline{y}_{i+1})\right)}{\gamma}  \right] \!\!
   \geqslant 4m. &
\end{flalign*}
Consequently
\begin{equation}
\left\|H^{-1}\right\|\leqslant \dfrac{1}{m}.
\label{egz5}
\end{equation}
Using Hadamard's theorem (see e.g. Theorem 5.3.10 from \cite{OR}), we get that $F$ is an homeomorphism. Since clearly $\mathbb{R}^{N+1}$ is non-empty and $0$ is the only image of the mapping $F$, we have that \eqref{konst18a}--\eqref{konst18c} has a unique solution. \\
The proof of second part of the Theorem \ref{teo21} is based on a part of the proof of Theorem 3 from \cite{H1}. We have that
$\displaystyle  Fu-Fv=\left(F'w\right)(u-v)
$
for some $w=\left(w_0,w_1,\ldots,w_N\right)^T\in\mathbb{R}^{N+1}$. Therefore
$\displaystyle
   u-v=(F'w)^{-1}\left( Fu-Fv\right)$
and finally due to inequality \eqref{egz5} we have that
\begin{equation*}
  \left\| u-v\right\|=\left\|(F'w)^{-1}(Fu-Fv)
        \right\|\leqslant\frac{1}{m}\left\|Fu-Fv\right\|.
\end{equation*} \qed
\end{proof}
\section{Mesh construction} \label{sec3}
Since the solution of the problem \eqref{uvod1}--\eqref{uvod3} changes rapidly near  $x=0$ and $x=1$, the mesh has to be refined there. Various meshes have been proposed by various authors. The most frequently analyzed are the exponentially graded meshes of Bakhvalov, see \cite{B1}, and piecewise uniform meshes of Shishkin, see \cite{SH1}.

Here we use the smoothed Shishkin mesh from \cite{Z1} and we construct it as follows. Let $N+1$ be the number of mesh points and $q\in(0,1/2)$ and $\sigma>0$ are mesh parameters. Define the Shishkin mesh transition point by
\begin{equation*}
\lambda:=\min\left\{\frac{\sigma\epsilon}{\sqrt{m}}\ln N,q\right\}.
\label{mreza1}
\end{equation*}
Let us chose $\sigma=2.$
\begin{remark}
For simplicity in representation, we assume that $\lambda=2\epsilon (\sqrt{m})^{-1}\ln N$, as otherwise the problem can be analyzed in the classical way. We shall also assume that $q N$ is an integer. This is easily achieved by choosing $q=1/4$ and $N$ divisible by 4 for example.
\end{remark}
The mesh $\Delta:x_0<x_1<\cdots<x_N$ is generated by $x_i=\varphi(i/N)$ with the mesh generating function
\begin{equation}
\varphi(t):=\left\{
              \begin{array}{ll}
                 \tfrac{\lambda}{q}t &t\in[0,q],\\
                 p(t-q)^3  +\frac{\lambda}{q}t          &t\in[q,1/2],\\
                 1-\varphi(1-t)         &t\in[1/2,1],
              \end{array}
      \right.
\label{mreza2}
\end{equation}
where $p$ chosen such that $\varphi(1/2)=1/2,$ i.e. $p=\tfrac{1}{2}(1-\tfrac{\lambda}{q})(\tfrac{1}{2}-q)^{-3}.$  \\ Note that $\varphi\in C^{1}[0,1]$ with $\left\|\varphi'\right\|_{\infty},\left\|\varphi''\right\|_{\infty}\leqslant C.$  Therefore we have that the mesh sizes $h_{i}=x_{i+1}-x_{i},\,i=0,1,\ldots,N-1$ satisfy
\begin{eqnarray}
\label{mreza31}
h_i&=&\int_{i/N}^{(i+1)/N}{\varphi'(t) dt}\leqslant C N^{-1}, \\
|h_{i+1}-h_i|&=&\left|\int_{i/N}^{(i+1)/N} \int_{t}^{t+1/N}{\varphi''(s) ds}\right|\leqslant CN^{-2}.
\label{mreza32}
\end{eqnarray}
\section{Uniform convergence}\label{sec4}
In this section we prove the theorem on $\epsilon$-uniform convergence of the discrete problem \eqref{konst18a}--\eqref{konst18c}.
The proof uses the decomposition of the solution $y$ to the problem
(\ref{uvod1})--(\ref{uvod2}) to the layer $s$ and a regular component $r$ given by
\begin{theorem} \cite{V1}
The solution $y$ to problem (\ref{uvod1})--(\ref{uvod2}) can be represented as
\begin{equation*}
   y=r+s,
\end{equation*}
where for $j=0,1,\ldots,k+2$ and $x\in[0,1]$ we have that
\begin{eqnarray}
\left|r^{(j)}(x)\right|&\leqslant& C,
\label{regularna} \\
\left|s^{(j)}(x)\right|&\leqslant& C \epsilon^{-j}\left(e^{-\frac{x}{\epsilon}\sqrt{m}}+e^{-\frac{1-x}{\epsilon}\sqrt{m}}\right).
\label{slojna}
\end{eqnarray}
\label{teorema2}
\end{theorem}
\begin{remark}\label{remark2}
Note that $e^{-\frac{x}{\epsilon}\sqrt{m}}\geqslant e^{-\frac{1-x}{\epsilon}\sqrt{m}}$ for $x\in[0,1/2]$ and $e^{-\frac{x}{\epsilon}\sqrt{m}}\leqslant e^{-\frac{1-x}{\epsilon}\sqrt{m}}$ for $x\in[1/2,1].$ These inequalities and the estimate \eqref{slojna} imply that the analysis of the error value can be done on the part of the mesh which corresponds to  $x\in[0,1/2]$ omitting the function $e^{-\frac{1-x}{\epsilon}\sqrt{m}}$, keeping in mind that on this part of the mesh we have that $h_{i-1}\leqslant h_i.$ An analogous analysis would hold for the part of the mesh which corresponds to $x\in[1/2,1]$ but with the omision of the function $e^{-\frac{x}{\epsilon}\sqrt{m}}$ and using the inequality  $h_{i-1}\geqslant h_i.$
\end{remark}
From here on in we use $\epsilon^2y''(x_k)=f(x_k,y(x_k)),\,k\in\left\{i-1,i,i+1\right\},$ and
\begin{align}
y_{i-1}-y_i=&-y'_ih_{i-1}+\frac{y''_i}{2}h^2_{i-1}-\frac{y'''_i}{6}h^3_{i-1}+\frac{y^{(iv)}(\zeta^{-}_{i-1})}{24}h^4_{i-1}, \label{teo121}\\
y_i-y_{i+1}=&-y'_ih_i-\frac{y''_i}{2}h^2_i-\frac{y'''_i}{6}h^3_i-\frac{y^{(iv)}(\zeta^{+}_i)}{24}h^4_{i},\label{teo122} \\
y''_{i-1}=& \ y''_i-y'''_i h_{i-1}+\frac{y^{(iv)}(\xi^{-}_{i-1})}{2}h^2_{i-1}, \label{teo2-1} \\
y''_{i+1}=&\ y''_i+y'''_ih_i+\frac{y^{(iv)}(\xi^{+}_i)}{2}h^2_i, \label{teo2-2}
\end{align}
where $\zeta^{-}_{i-1},\,\xi^{-}_{i-1}\in(x_{i-1},x_i),\,\zeta^{+}_i,\,\xi^{+}_{i}\in(x_{i},x_{i+1}).$
We begin with a lemma that will be used further on in the proof on the uniform convergence.
\begin{lemma}On the part of the modified Shishkin mesh \eqref{mreza2} where $x_i,x_{i\pm 1}\in\left[x_{N/4-1},\lambda\right]$ $\cup\left[\lambda,1/2\right]$, assuming that  $\displaystyle \epsilon\leqslant\tfrac{C}{N}$, for $i=\frac{N}{4},\ldots,\frac{N}{2}-1$ we have the following estimate
\begin{equation} \label{lemma11}
     \left({\frac{\cosh(\beta h_{i-1})-1}{\gamma\sinh(\beta h_{i-1})}+\frac{\cosh(\beta h_{i})-1}{\gamma\sinh(\beta h_{i})}}\right)^{-1}
      \left|\frac{y_{i-1}-y_i}{\sinh(\beta h_{i-1})}-\frac{y_i-y_{i+1}}{\sinh(\beta h_i)}\right|\leqslant\frac{C}{N^2}.
\end{equation}
\label{lema1}
\end{lemma}
\upshape
\begin{proof}
We are using the decomposition from Theorem \ref{teorema2} and expansions \eqref{teo2-1}, \eqref{teo2-2}. For the regular component $r$ we have that

\begin{flalign}
\nonumber & \left(\frac{\cosh(\beta h_{i-1})-1}{\gamma\sinh(\beta h_{i-1})}+\frac{\cosh(\beta h_{i})-1}{\gamma\sinh(\beta h_{i})}\right)^{-1}
      \left|\frac{r_{i-1}-r_i}{\sinh(\beta h_{i-1})}-\frac{r_i-r_{i+1}}{\sinh(\beta h_i)}\right|  \\
      \nonumber
   & \hspace{0.7cm}\leqslant   \gamma \left|r'_i\frac{\displaystyle \beta h_{i-1}h_i\sum_{n=1}^{+\infty} {\frac{\displaystyle \beta^{2n}(h^{2n}_i-h^{2n}_{i-1})}{(2n+1)!}}}{\displaystyle \sum_{n=1}^{+\infty}\frac{(\beta h_i)^{2n}}{(2n)!}\sinh(\beta h_{i-1})} \right| \nonumber \\
 & \hspace{0.7cm}  +\gamma\left| \frac{\frac{r''(\mu^{+}_{i})}{2}\displaystyle\sum_{n=0}^{+\infty}{\frac{(\beta h_{i-1})^{2n+1} }{(2n+1)!}h^2_i} }
       {(\cosh(\beta h_i)-1)\sinh(\beta h_{i-1})}\right| +
        \gamma \left|\frac{ \frac{r''(\mu^{-}_{i})}{2}\displaystyle\sum_{n=0}^{+\infty}{\frac{(\beta h_{i})^{2n+1} }{(2n+1)!}h^2_{i-1}}}
       {(\cosh(\beta h_i)-1)\sinh(\beta h_{i-1})}\right|. \label{15} &
\end{flalign}
First we want to estimate the expressions containing only the first derivatives in the RHS of inequality (\ref{15}).
From the identity $\displaystyle
 a^n-b^n=(a-b)(a^{n-1}+a^{n-2}b+\ldots+ab^{n-2}+b^{n-1}),\,n\in\mathbb{N},
$
and the inequalities $h_{i-1}\leqslant h_i$, $i=1,\ldots,\frac{N}{2}-1$, we get the inequality
$\displaystyle
   h^n_i-h^n_{i-1}
   \leqslant n(h_i-h_{i-1})h^{n-1}_i$,
 which yields that $\forall n\in\mathbb{N}$,
\begin{equation}
  \frac{\beta^{2n}(h^{2n}_i-h^{2n}_{i-1})}{(2n+1)!}
                      <\frac{\beta^{2n}(h^2_i-h^2_{i-1})h^{2(n-1)}_i}{(2n)!}.
\label{17}
\end{equation}
Using inequality (\ref{17})  together with
(\ref{regularna}), we get that

\begin{equation}
   \gamma \left|r'_i\frac{\beta h_{i-1}h_i\displaystyle\sum_{n=1}^{+\infty}{\frac{\beta^{2n}(h^{2n}_i-h^{2n}_{i-1})}{(2n+1)!}}}
                                     {\displaystyle\sum_{n=1}^{+\infty}\frac{(\beta h_i)^{2n}}{(2n)!}\sinh(\beta h_{i-1})} \right|
          \leqslant C(h_i-h_{i-1}).
\label{18}
\end{equation}
Now we want to estimate the terms containing the second derivatives from the RHS of (\ref{15}). Using inequality (\ref{regularna}), after some simplification, we get that
\begin{equation}
   \left|\frac{ \frac{r''(\mu^{+}_{i})}{2}\displaystyle\sum_{n=0}^{+\infty}{\frac{(\beta h_{i-1})^{2n+1} }{(2n+1)!}h^2_{i}}}
       {(\cosh(\beta h_i)-1)\sinh(\beta h_{i-1})}\right|
    \leqslant\left| \frac{r''(\mu^{+}_{i})h^2_{i}}{\beta^2 h^2_i}\right|    \leqslant C\epsilon^2 ,
\label{20}
\end{equation}
\begin{equation}
    \left| \frac{\frac{r''(\mu^{-}_{i})}{2}\displaystyle\sum_{n=0}^{+\infty}{\frac{(\beta h_{i})^{2n+1} }{(2n+1)!}h^2_{i-1}} }
       {(\cosh(\beta h_i)-1)\sinh(\beta h_{i-1})}\right|
    \leqslant C(\epsilon^2+h_{i-1}h_i)     .
\label{21}
\end{equation}
For the layer component $s$, first we have that
\begin{flalign} \nonumber
&\left({\frac{\cosh(\beta h_{i-1})-1}{\gamma\sinh(\beta h_{i-1})}+\frac{\cosh(\beta h_{i})-1}{\gamma\sinh(\beta h_{i})}}\right)^{-1}
    \left|\frac{s_{i-1}-s_i}{\sinh(\beta h_{i-1})}-\frac{s_i-s_{i+1}}{\sinh(\beta h_i)}\right| & &\\        &  \hspace{1.2cm}\leqslant\gamma\left| \frac{\beta h_{i}(s_{i-1}-s_i)-\beta h_{i-1}(s_{i}-s_{i+1})}
                              {\displaystyle\sum_{n=1}^{+\infty}{\frac{(\beta h_{i})^{2n}}{(2n)!}}\sinh(\beta h_{i-1})}\right| & & \nonumber\\
 & \hspace{1.2cm} +\gamma\left| \frac{\beta h_i\displaystyle\sum_{n=1}^{+\infty}{\frac{(\beta h_{i})^{2n}}{(2n+1)!}}(s_{i-1}-s_i)
                              -\beta h_{i-1}\displaystyle\sum_{n=1}^{+\infty}{\frac{(\beta h_{i-1})^{2n}}{(2n+1)!}}(s_{i}-s_{i+1})}
                              {\displaystyle\sum_{n=1}^{+\infty}{\frac{(\beta h_{i})^{2n}}{(2n)!}}\sinh(\beta h_{i-1})}\right|. & &
\label{22}
\end{flalign}
The first term of the RHS of (\ref{22}) can be bounded by

\begin{flalign} \nonumber
   & \gamma\left|\frac{\beta h_{i}\left[-s'_i h_{i-1}+\frac{s''(\mu^{-}_{i})}{2}h^2_{i-1}\right]
                              -\beta h_{i-1}\left[-\left(s'_i h_i+\frac{s''(\mu^{+}_i)}{2}h^2_i\right)\right]}
                                      {\frac{\beta^2 h^2_i}{2}\beta h_{i-1} }\right|  &  \\
   & \hspace{5cm}\leqslant\epsilon^2\left| s''(\mu^{-}_{i})\frac{h_{i-1}}{h_i}+s''(\mu^{+}_i)\right|
                                              \leqslant \frac{C}{N^2}.
\label{23}
\end{flalign}
For the second term of the RHS of (\ref{22}) we get that
\begin{flalign}\nonumber
   &\gamma\left| \frac{\beta h_i\displaystyle\sum_{n=1}^{+\infty}{\frac{(\beta h_{i})^{2n}}{(2n+1)!}}(s_{i-1}-s_i)
                              -\beta h_{i-1}\displaystyle\sum_{n=1}^{+\infty}{\frac{(\beta h_{i-1})^{2n}}{(2n+1)!}}(s_{i}-s_{i+1})}
                              {\displaystyle\sum_{n=1}^{+\infty}{\frac{(\beta h_{i})^{2n}}{(2n)!}}\sinh(\beta h_{i-1})}\right|  \\
  &\leqslant\gamma\ \frac{\beta h_i}{\sinh(\beta h_{i-1})}
         \frac{\displaystyle\sum_{n=1}^{+\infty}{\frac{(\beta h_{i})^{2n}}{(2n+1)!}}}
             {\displaystyle\sum_{n=1}^{+\infty}{\frac{(\beta h_{i})^{2n}}{(2n)!}}}   |s_{i-1}-s_i|
         + \gamma\ \frac{\beta h_{i-1}}{\beta h_{i-1}}
        \frac{\displaystyle\sum_{n=1}^{+\infty}{\frac{(\beta h_{i-1})^{2n}}{(2n+1)!}}}
             {\displaystyle\sum_{n=1}^{+\infty}{\frac{(\beta h_{i})^{2n}}{(2n)!}}}   |s_{i}-s_{i+1}|,
\label{24}
\end{flalign}
\begin{equation}
  \begin{split}
   \gamma \dfrac{\sum_{n=1}^{+\infty}{\dfrac{(\beta h_{i-1})^{2n}}{(2n+1)!}}}
             {\sum_{n=1}^{+\infty}{\dfrac{(\beta h_{i})^{2n}}{(2n)!}}}   |s_{i}-s_{i+1}|\leqslant \dfrac{C}{N^2}.
  \end{split}
\label{24a}
\end{equation}
In the first expression  of the RHS of (\ref{24})
we have the term $\displaystyle \frac{\beta h_i}{\sinh(\beta h_{i-1})}.$ Although this ratio is bounded by $\displaystyle \frac{h_i}{h_{i-1}}$, this quotient is not bounded for $x_{i}=\lambda$ when $\epsilon\rightarrow 0.$ This is why we are going to estimate this expression separately on the transition part and on the nonequidistant part of the mesh. In the case $i=\frac{N}{4}$, using the fact that
 $\sum_{n=1}^{+\infty}{\frac{x^{2n}}{(2n)!}}=\cosh x-1$ and $\sum_{n=1}^{+\infty}{\frac{x^{2n+1}}{(2n+1)!}}=\sinh x-x,\,\forall x\in\mathbb{R}$
and the fact that the function $ r(x)=\frac{\sinh x-x}{\cosh x-1}$ takes values from the interval $(0,1)$ when $x>0$, we have that
\begin{equation}
   \gamma\left| \frac{\beta h_i\displaystyle\sum_{n=1}^{+\infty}{\frac{(\beta h_{i})^{2n}}{(2n+1)!}}(s_{i-1}-s_i) }
                              {\displaystyle\sum_{n=1}^{+\infty}{\frac{(\beta h_{i})^{2n}}{(2n)!}}\sinh(\beta h_{i-1})}\right|
      \leqslant  \gamma\frac{|s_{i-1}-s_i|}{\sinh(\beta h_{i-1})}
            \leqslant \frac{C}{N^2}.
\label{26}
\end{equation}
When $i=\frac{N}{4}+1,\ldots,\frac{N}{2}-1$, we can use
  $ \tfrac{\sum_{n=1}^{+\infty}{\tfrac{x^{2n}}{(2n+1)!}}}{\sum_{n=1}^{+\infty}{\tfrac{x^{2n}}{(2n)!}}}=
                                                            \tfrac{\sinh x-x}{x(\cosh x-1 )}=p(x)$ and   \\ $0<p(x)<\frac{1}{3}$ for $x>0.$
Therefore
\begin{equation}
    \gamma\left| \frac{\beta h_i\displaystyle\sum_{n=1}^{+\infty}{\frac{(\beta h_{i})^{2n}}{(2n+1)!}}(s_{i-1}-s_i) }
                              {\displaystyle\sum_{n=1}^{+\infty}{\frac{(\beta h_{i})^{2n}}{(2n)!}}\sinh(\beta h_{i-1})}\right|
   \leqslant \gamma   \frac{\beta h_{i}}{\beta h_{i-1}}\frac{\displaystyle\sum_{n=1}^{+\infty}{\frac{(\beta h_i)^{2n}}{(2n+1)!}}}
                                                           {\displaystyle\sum_{n=1}^{+\infty}{\frac{(\beta h_i)^{2n}}{(2n)!}}}
                                                          |s_{i-1}-s_i|                      \leqslant\frac{C}{N^2}.
\label{27}
\end{equation}
Using equations (\ref{mreza31}), (\ref{mreza32}) and (\ref{18})--(\ref{27}), we complete the proof of the lemma. \qed
\end{proof}

Now we state the main theorem on $\epsilon-$uniform convergence of our difference scheme and specially chosen layer-adapted mesh.
\begin{theorem} The discrete problem \eqref{konst18a}--\eqref{konst18c} on the mesh from Section \ref{sec2}. is uniformly convergent with respect to $\epsilon$  and
\begin{equation*}
\max_{i}\left|y_i-\overline{y}_i\right|\leq C\left\{ \begin{array}{l}
\dfrac{\ln^2 N}{N^2},\:i=0,\ldots,\tfrac{N}{4}-1\vspace{.25cm}\\
\dfrac{1}{N^2},\:i=\tfrac{N}{4},\ldots,\tfrac{3N}{4}\vspace{.25cm} \\
\dfrac{\ln^2 N}{N^2},\:i=\tfrac{3N}{4}+1,\ldots,N,
\end{array}\right.
%\label{teo1}
\end{equation*}
where $y$ is the solution of the problem (\ref{uvod1}), $\overline{y}$ is the corresponding numerical solution of (\ref{konst18a})--\eqref{konst18c} and $C>0$ is a constant independent of $N$ and $\epsilon$.
\end{theorem}
\begin{proof}
We shall use the technique from \cite{V2}, i.e. since we have stability from Theorem \ref{teo21}, we have that $\left\|y-\overline{y}\right\|\leqslant C\left\|Fy-F\overline{y}\right\|$ and since \eqref{konst18a}--\eqref{konst18c} implies that  $F\overline{y}=0$, it only remains to estimate $\left\|Fy\right\|$.

Let  $i=0,1,\ldots,\frac{N}{4}-1$. The discrete problem (\ref{konst18a})--\eqref{konst18c} can be written down on this part of the mesh in the following form
\begin{flalign*}
  F_0y=& \ 0, & &\\
   F_iy
      =&\frac{\gamma}{\Delta d_i+\Delta d_{i+1}}{\Big[}\left(3a_i+d_i+\Delta d_{i+1} \right)(y_{i-1}-y_i)
                    -\left[3a_{i+1}+d_{i+1}+\Delta d_{i} \right](y_{i}-y_{i+1})\\
       &\hspace{2cm}   \left. -\frac{f(x_{i-1},y_{i-1})+2f(x_{i},y_{i})+f(x_{i+1},y_{i+1}) }{\gamma}
                                        \left(\Delta d_{i}+\Delta d_{i+1}\right)\right]\\
      =&\frac{\gamma}{2\Delta d_i}\Big[   \left(3a_i+d_i+\Delta d_i \right)\left(y_{i-1}-y_i-(y_i-y_{i+1})\right) \\
      &\hspace{2cm}   \left. -2\frac{f(x_{i-1},y_{i-1})+2f(x_{i},y_{i})+f(x_{i+1},y_{i+1}) }{\gamma}\Delta d_{i}     \right] \\
      =&\frac{\gamma}{2(\cosh(\beta h_i)-1)} {\big[} \left(2+2\cosh(\beta h_i)\right)\left(y_{i-1}-y_i-(y_i-y_{i+1})\right) \\
     & \hspace{2cm} \left.
             -2\frac{f(x_{i-1},y_{i-1})+2f(x_{i},y_{i})+f(x_{i+1},y_{i+1})}{\gamma}(\cosh(\beta h_i)-1)\right],
\end{flalign*}
for $i=1,2,\ldots,\frac{N}{4}-1$. Using the expansions \eqref{teo121} and \eqref{teo122}, we get that
\begin{flalign*}
   F_iy
       =&\frac{\gamma}{\beta^2h^2_i+2\mathcal{O}\left(\beta^4 h^4_i\right)} \\
       & \hspace{.75cm}
        \left[ \left(4+\beta^2h^2_i+2\mathcal{O}\left( \beta^4h^4_i\right) \right)
                    \left(y''_ih^2_i+\frac{y^{(iv)}\left(\zeta^{-}_{i-1}\right)
                    +y^{(iv)}\left(\zeta^{+}_i\right)}{24}h^4_i \right)\right. &\\
        &\hspace{.75cm}-\frac{1}{\beta^2}\left.\left(4y''_i
                               +\frac{y^{(iv)}\left(\xi^{-}_{i-1}\right)+y^{(iv)}\left(\xi^{+}_i\right)}{2}h^2_i \right)
              \left( \beta^2h^2_i+2\mathcal{O}\left(\beta^4h^4_i \right) \right)    \right]
                      \end{flalign*}
              \begin{flalign*}
       =&\frac{\gamma}{\beta^2h^2_i+2\mathcal{O}\left(\beta^4 h^4_i\right)}
        \left[ 4\frac{y^{(iv)}\left(\zeta^{-}_{i-1}\right)+y^{(iv)}\left(\zeta^{+}_i\right)}{24}h^4_i \right. \\
        +& \left.
          \left(\beta^2h^2_i+2\mathcal{O}\left(\beta^4h^4_i \right) \right)
           \left( y''_ih^2_i+\frac{y^{(iv)}\left(\zeta^{-}_{i-1}\right)+y^{(iv)}\left(\zeta^{+}_i\right)}{24}h^4_i\right)
           -\frac{8}{\gamma}\epsilon^2y''_i\mathcal{O}\left( \beta^4h^4_i\right)\right.\\
       -&\left.
                      \frac{\epsilon^2\left(y^{(iv)}\left(\xi^{-}_{i-1}\right)+y^{(iv)}\left(\xi^{+}_i\right) \right)}{2\gamma}  h^2_i
                      \left[\beta^2h^2_i+2\mathcal{O}\left( \beta^4h^4_i\right) \right] \right],
\end{flalign*}
for $i=1,\ldots,\frac{N}{4}-1$ and hence
$\displaystyle
   \left|F_iy \right|\leqslant\frac{C\ln^2N}{N^2},$
for $i=0,1,\ldots,\frac{N}{4}-1$. \\
Now let $i=\frac{N}{4},\ldots,\frac{N}{2}-1.$  We rewrite equations (\ref{konst18a})--\eqref{konst18c} as
\begin{flalign*}
   F_iy=&\frac{\gamma}{\Delta d_i+\Delta d_{i+1}}
                {\big[} (\Delta d_i+\Delta d_{i+1})\left(y_{i-1}-y_i-(y_i-y_{i+1}) \right) & \\
                &\hspace{2cm}+4\left( a_i(y_{i-1}-y_i)-a_{i+1}(y_i-y_{i+1})\right) \\
                &\hspace{2cm}- \left.\frac{f(x_{i-1},y_{i-1})+2f(x_{i},y_{i})+f(x_{i+1},y_{i+1})}{\gamma}(\Delta d_i+\Delta d_{i+1})\right].
\end{flalign*}
We estimate the linear and the nonlinear term separately. For the nonlinear term we get
\begin{flalign*}
& \frac{\gamma}{\Delta d_i+\Delta d_{i+1}}
 \left| \frac{f(x_{i-1},y_{i-1})+2f(x_{i},y_{i})+f(x_{i+1},y_{i+1}) }{\gamma}
       \left(\Delta d_{i}+\Delta d_{i+1}\right)\right| & &\\
 & \hspace{1cm}= \left|f(x_{i-1},y_{i-1})+2f(x_{i},y_{i})+f(x_{i+1},y_{i+1})\right|=\epsilon^2\left|y''_{i-1}+2y''_{i}+y''_{i+1}\right| \leqslant\frac{C}{N^2}.
\end{flalign*}
For the linear term we get
\begin{flalign} \nonumber
&  \frac{\gamma}{\Delta d_i\!+\!\Delta d _{i+1}}
  \Big|\! (\Delta d_i+\Delta d_{i+1})\!\left[y_{i-1}-y_i-y_i+y_{i+1}) \right]\!+\!4\left[ a_i(y_{i-1}-y_i)\!-\!a_{i+1}(y_i-y_{i+1})\right]\!\!\Big|\\
& \hspace{0.3cm}             \leqslant \gamma \left| y_{i-1}-y_i-(y_i-y_{i+1})\right|+  \frac{4\gamma}{\Delta d_i+\Delta d _{i+1}}\left|a_i(y_{i-1}-y_i)-a_{i+1}(y_i-y_{i+1}) \right|.
\label{teo8}
\end{flalign}
For the first term in the RHS of (\ref{teo8}) we get
\begin{equation*}
  \begin{split}
   \gamma \left| y_{i-1}-y_i-(y_i-y_{i+1})\right|\leqslant\frac{C}{N^2},
  \end{split}
\end{equation*}
while for the second term in the RHS of (\ref{teo8}), using \eqref{lemma11} and \eqref{konst18aa}, we get that
\begin{equation*}
  \begin{split}
   \frac{4\gamma}{\Delta d_i+\Delta d _{i+1}}\left|a_i(y_{i-1}-y_i)-a_{i+1}(y_i-y_{i+1}) \right|\leqslant\frac{C}{N^2}.
  \end{split}
\end{equation*}
Hence, we get that
$\displaystyle
  \left|F_iy\right|\leqslant\frac{C}{N^2}
$
for $i=\frac{N}{4},\ldots,\frac{N}{2}-1$. \\
The proof for $i=\frac{3N}{4}+1,\ldots,N$ is analogous to the case $i=0,1,\ldots,\frac{N}{4}-1$ and the proof for $i=\frac{N}{2}+1,\ldots, \frac{3N}{4}$ is analogous to the case $i=\frac{N}{4},\frac{N}{2}-1$
in view of Remark \ref{remark2} and Lemma \ref{lema1}. Finally, the case  $i=\frac{N}{2}$ is simply shown since $h_{N/2-1}=h_{N/2}$, $e^{-\frac{x}{\epsilon}\sqrt{m}}<<\epsilon^2$ and  $e^{-\frac{1-x}{\epsilon}\sqrt{m}}<<\epsilon^2$ for $x\in[x_{N/2-1},x_{N/2+1}]$.
 \qed
\end{proof}
\section{Numerical results}\label{sec5}
In this section we present numerical results to confirm the uniform accuracy of the discrete problem (\ref{konst18a})--\eqref{konst18c}. To demonstrate the efficiency of the method, we present two examples having boundary layers. The problems from our examples have known exact solutions, so we calculate $E_N$ as
\begin{equation}
  E_N=\max_{0\leqslant i\leqslant N}\left|y(x_i)-\overline{y}^{N}(x_i) \right|,
\label{num1}
\end{equation}
where $\overline{y}^N(x_i)$ is the value of the numerical solutions at the mesh point $x_i$, where the mesh has $N$ subintervals, and $y(x_i)$ is the value of the exact solution at $x_i$. The rate of convergence  Ord is calculated using
\begin{equation*}
 \text{Ord}=\frac{\ln E_N-\ln E_{2N}}{\ln 	\frac{2k}{k+1}},
\end{equation*}
where $N=2^{k},\:k=6,7,\ldots,13.$  Tables 1 and 2 give the numerical results for our two examples and we can see that the theoretical and experimental results match.

\begin{example} \upshape Consider the following problem, see \cite{H2}
\begin{equation*}
 \epsilon^2y''=y+\cos^2(\pi x)+2(\epsilon \pi)^2\cos(2\pi x)\quad\text{for }x\in(0,1),\quad y(0)=y(1)=0.
\end{equation*}
The exact solution of this problem is given by
$\displaystyle
  y(x)=\frac{e^{-\frac{x}{\epsilon}}+e^{-\frac{1-x}{\epsilon}}}{1+e^{-\frac{1}{\epsilon}}}-\cos^2(\pi x).
$
\label{primjer1}
The nonlinear system was solved using the initial condition $y_0=-0.5$ and the value of the constant  $\gamma=1$.
\begin{center}
\begin{figure}[!h]
 \includegraphics[scale=.45]{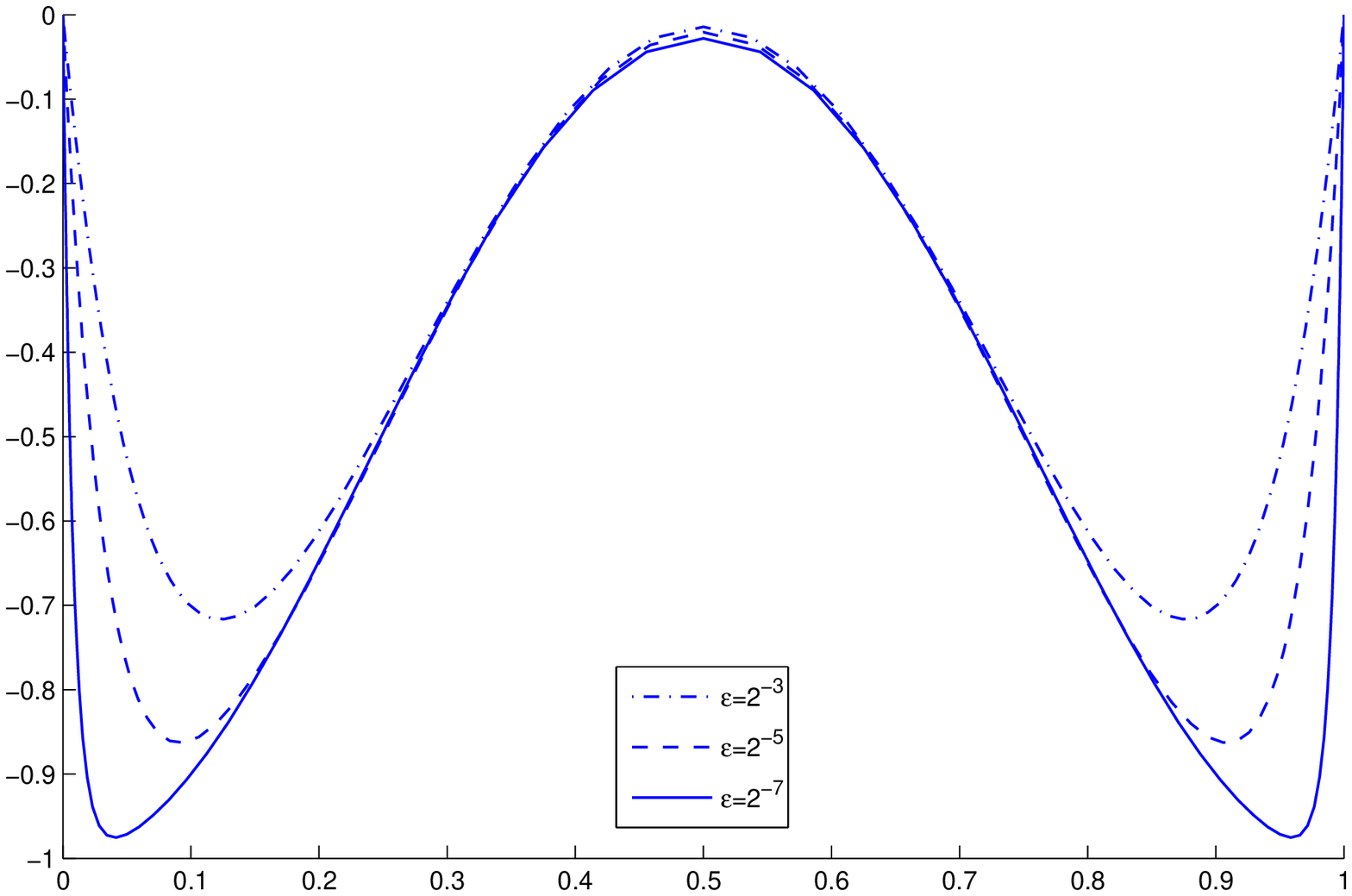}
 \caption{Numerical solution graphs from example \ref{primjer1} for values $\epsilon= 2^{-3},\,2^{-5},2^{-7}$.}
\end{figure}
\end{center}

\begin{table}[!h]\tiny
\centering
\begin{tabular}{c|cc|cc|cc|cc|cc}\hline
     $N$ &$E_n$&Ord&$E_n$&Ord&$E_n$&Ord&$E_n$&Ord&$E_n$&Ord\\\hline
$2^{6}$&$1.0212e-03$&$2.61$   &$2.8612e-03$&$2.01$ &$3.1123e-03$&$2.08$   &$4.3466e-03$&$2.08$  &$4.6523e-03$&$2.08$         \\
$2^{7}$&$2.5012e-04$&$2.23$   &$9.6837e-04$&$2.11$ &$1.0144e-03$&$2.07$   &$1.4166e-03$&$2.06$  &$1.5163e-03$&$2.04$         \\
$2^{8}$&$7.1810e-05$&$2.01$   &$2.9732e-04$&$2.09$ &$3.1849e-04$&$2.04$   &$4.4730e-04$&$2.05$  &$4.7876e-04$&$2.05$         \\
$2^{9}$&$2.2591e-05$&$2.03$   &$8.9328e-05$&$2.00$ &$9.8480e-05$&$2.00$   &$1.3752e-04$&$2.00$  &$1.4719e-04$&$2.02$         \\
$2^{10}$&$6.8505e-06$&$2.00$  &$2.7570e-05$&$2.00$ &$3.0395e-05$&$2.00$   &$4.2443e-05$&$2.00$  &$4.5428e-05$&$2.00$        \\
$2^{11}$&$2.0723e-06$&$2.00$  &$8.3400e-06$&$2.00$ &$9.1945e-06$&$2.00$   &$1.2839e-05$&$2.00$  &$1.3742e-05$&$2.00$         \\
$2^{12}$&$6.1654e-07 $&$2.00$ &$2.4813e-06 $&$2.00$ &$2.7356e-06$&$2.00$  &$3.8197e-06$&$2.00$  &$4.0885e-06$&$2.00$        \\
$2^{13}$&$1.8090e-07 $&$-$    &$7.2803e-07 $&$-$   &$8.0262e-07$&$-$      &$1.1208e-06$&$-$     &$1.1996e-06$&$-$         \\\hline
 $\epsilon$&\multicolumn{2}{c}{$2^{-3}$}&\multicolumn{2}{c}{$2^{-5}$}
            &\multicolumn{2}{c}{$2^{-7}$}&\multicolumn{2}{c}{$2^{-10}$}
             &\multicolumn{2}{c}{$2^{-15}$}\\\hline\hline
     $N$ &$E_n$&Ord&$E_n$&Ord&$E_n$&Ord&$E_n$&Ord&$E_n$&Ord\\\hline
$2^{6}$  &$4.6579e-03$&$2.08$  &$4.6579e-03$&$2.08$  & $4.6579e-03$&$2.08$  &$4.6579e-03$&$2.08$     &$4.6796e-03 $&$2.06 $   \\
$2^{7}$  &$1.5181e-03$&$2.04$  &$1.5181e-03$&$2.04$  & $1.5181e-03$&$2.04$  &$1.5181e-03$&$2.04$     &$1.5417e-03 $&$2.02 $   \\
$2^{8}$  &$4.7934e-04$&$2.05$  &$4.7934e-04$&$2.05$  & $4.7934e-04$&$2.05$  &$4.7934e-04$&$2.05$     &$4.9781e-03 $&$2.03 $   \\
$2^{9}$  &$1.4736e-04$&$2.02$  &$1.4736e-04$&$2.02$  &$1.4736e-04$&$2.02$   &$1.4736e-04$&$2.02$     &$1.5481e-04 $&$2.00 $   \\
$2^{10}$ &$4.5483e-05$&$2.00$  &$4.5483e-05$&$2.00$  &$4.5483e-05$&$2.00$   &$4.5483e-05$&$2.00$     &$4.7782e-05 $&$2.00 $   \\
$2^{11}$ &$1.3758e-05$&$2.00$  &$1.3758e-05$&$2.00$  &$1.3758e-05$&$2.00$   &$1.3758e-05$&$2.00$     &$1.4454e-05 $&$2.00 $   \\
$2^{12}$ &$4.0934e-06$&$2.00$  &$4.0934e-06$&$2.00$  &$4.0934e-06$&$2.00$   &$4.0934e-06$&$2.00$     &$4.3004e-06 $&$2.00 $   \\
$2^{13}$ &$1.2010e-06$&$-$     &$1.2010e-06$&$-$     &$1.2010e-06$&$-$      &$1.2010e-06$&$-$        &$1.2617e-06$&$-$      \\\hline
 $\epsilon$&\multicolumn{2}{c}{$2^{-25}$}
            &\multicolumn{2}{c}{$2^{-30}$}&\multicolumn{2}{c}{$2^{-35}$}
             &\multicolumn{2}{c}{$2^{-40}$}&\multicolumn{2}{c}{$2^{-45}$}\\\hline
\end{tabular}
\caption{Error $E_N$ and convergence rates Ord for approximate solution for example \ref{primjer1}.}
\end{table}
\end{example}
\newpage
\begin{example} \upshape Consider the following problem
\begin{equation*}
  \epsilon^2 y''=(y-1)(1+(y-1)^2)+g(x)\quad \text{for }x\in(0,1),\quad y(0)=(1)=0,
\end{equation*}
where
$\displaystyle
   g(x)=\frac{\cosh^3\frac{1-2x}{2\epsilon^3}}{\cosh^3\frac{1}{2\epsilon}} .
$
The exact solution of this problem is given by
$\displaystyle
  y(x)=1-\frac{e^{-\frac{x}{\epsilon}}+e^{-\frac{1-x}{\epsilon}}}{1+e^{-\frac{1}{\epsilon}}}.
$
\label{primjer2}
The nonlinear system was solved using the initial guess $y_0=1$. The exact solution implies that $0\leqslant y\leqslant 1,\:\forall x\in[0,1],$ so the value of the constant $\gamma=4$ was chosen so that we have $\gamma\geqslant f_y(x,y),\:\forall(x,y)\in[0,1]\times[0,1]\subset  [0,1]\times\mathbb{R}$.

\begin{center}
\begin{figure}[!h]
 \includegraphics[scale=.55]{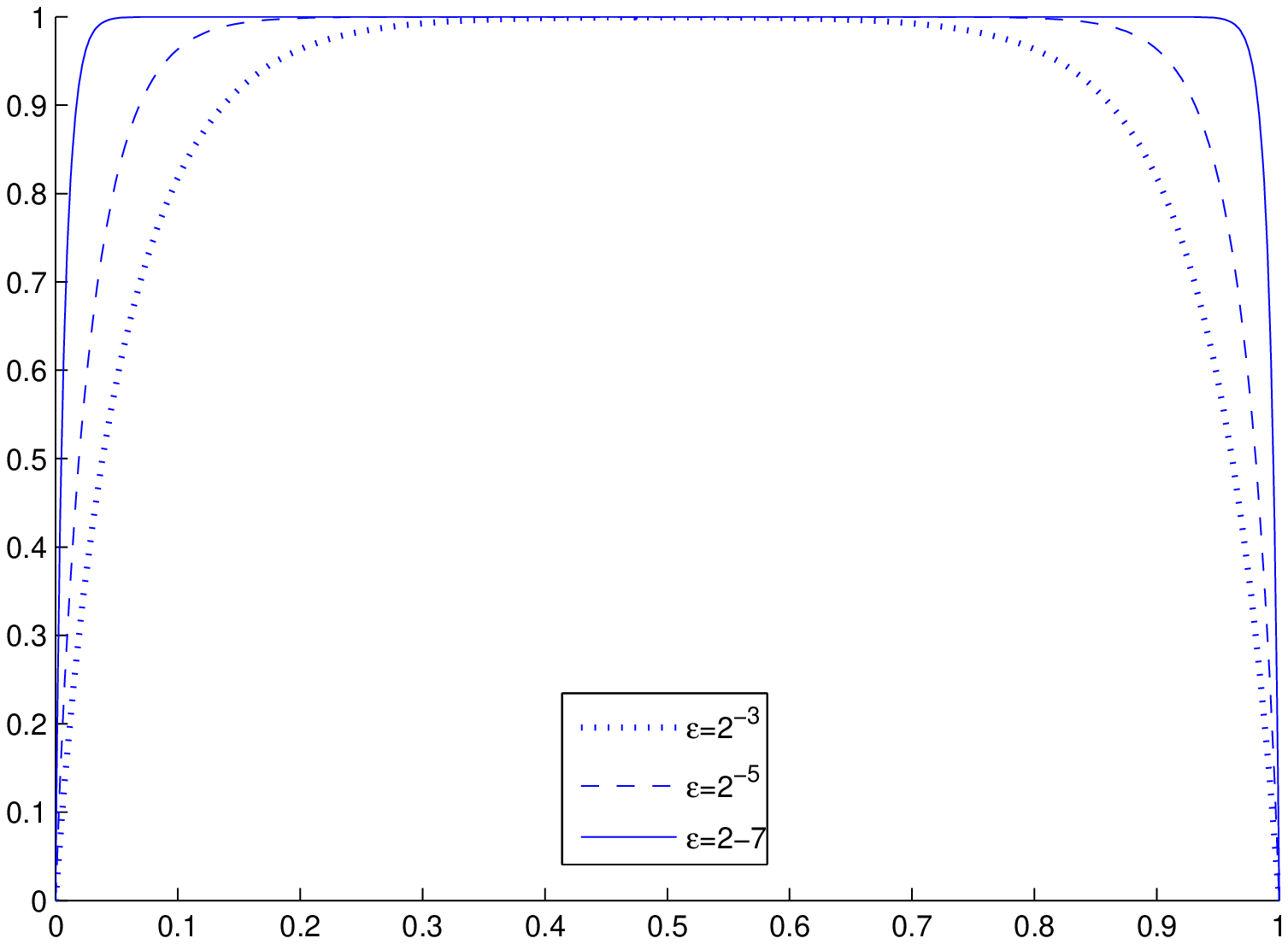}
 \caption{Numerical solution graphs from example \ref{primjer2} for values $\epsilon= 2^{-3},\,2^{-5},2^{-7}$.}
\end{figure}
\end{center}

\begin{table}[!h]\tiny
\centering
\begin{tabular}{c|cc|cc|cc|cc|cc}\hline
     $N$ &$E_n$&Ord&$E_n$&Ord&$E_n$&Ord&$E_n$&Ord&$E_n$&Ord\\\hline
$2^{6}$&$1.7568e-03$&$2.45$   &$3.0164e-03$&$1.98$ &$3.1822e-03$&$2.08$   &$4.6272e-03$&$2.08$  &$6.7583e-03$&$2.08$         \\
$2^{7}$&$4.6905e-04$&$2.33$   &$1.0375e-03$&$2.18$ &$1.0371e-03$&$2.17$   &$1.5081e-03$&$2.06$  &$2.2026e-03$&$2.04$         \\
$2^{8}$&$1.2733e-04$&$1.99$   &$3.0632e-04$&$2.24$ &$3.0792e-04$&$2.24$   &$4.7617e-04$&$2.09$  &$7.0331e-04$&$2.05$         \\
$2^{9}$&$4.0521e-05$&$2.00$   &$8.4422e-05$&$2.00$ &$8.4863e-05$&$2.00$   &$1.4306e-04$&$2.04$  &$2.1622e-04$&$2.02$         \\
$2^{10}$&$1.2507e-05$&$2.00$  &$2.6056e-05$&$2.00$ &$2.6192e-05$&$2.00$   &$4.3129e-05$&$2.00$  &$6.5955e-05$&$2.00$        \\
$2^{11}$&$3.7832e-06$&$2.00$  &$7.8820e-06$&$2.00$ &$7.9231e-06$&$2.00$   &$1.3046e-05$&$2.00$  &$1.9951e-05$&$2.00$         \\
$2^{12}$&$1.1256e-06 $&$2.00$ &$2.3451e-06 $&$2.00$ &$2.3573e-06$&$2.00$  &$3.8816e-06$&$2.00$  &$5.9356e-06$&$2.00$        \\
$2^{13}$&$3.3025e-07 $&$-$    &$6.8805e-07 $&$-$   &$6.9164e-07$&$-$      &$1.1389e-06$&$-$     &$1.7416e-06$&$-$         \\\hline
 $\epsilon$&\multicolumn{2}{c}{$2^{-3}$}&\multicolumn{2}{c}{$2^{-5}$}
            &\multicolumn{2}{c}{$2^{-7}$}&\multicolumn{2}{c}{$2^{-10}$}
             &\multicolumn{2}{c}{$2^{-15}$}\\\hline\hline
     $N$ &$E_n$&Ord&$E_n$&Ord&$E_n$&Ord&$E_n$&Ord&$E_n$&Ord\\\hline
$2^{6}$  &$6.7592e-03$&$2.08$ &$6.7592e-03$&$2.08$     &$6.7592e-03$&$2.08$ &$6.7592e-03$&$2.08$    &$6.8012e-03 $&$2.08 $   \\
$2^{7}$  &$2.2029e-03$&$2.04$ &$2.2029e-03$&$2.04$     &$2.2029e-03$&$2.04$ &$2.2029e-03$&$2.04$    &$2.2166e-03 $&$2.02 $   \\
$2^{8}$  &$7.0340e-04$&$2.05$ &$7.0340e-04$&$2.05$     &$7.0340e-04$&$2.05$ &$7.0340e-04$&$2.05$    &$7.1574e-03 $&$2.01 $   \\
$2^{9}$  &$2.1625e-04$&$2.02$ &$2.1625e-04$&$2.02$     &$2.1625e-04$&$2.02$ &$2.1625e-04$&$2.02$    &$2.2516e-04 $&$1.99 $   \\
$2^{10}$ &$6.5974e-05$&$2.00$ &$6.5974e-05$&$2.00$     &$6.5974e-05$&$2.00$ &$6.5974e-05$&$2.00$    &$6.9905e-05 $&$2.00 $   \\
$2^{11}$ &$1.9954e-05$&$2.00$ &$1.9954e-05$&$2.00$     &$1.9954e-05$&$2.00$ &$1.9954e-05$&$2.00$    &$2.1146e-05 $&$2.00 $   \\
$2^{12}$ &$5.9367e-06$&$2.00$ &$5.9367e-06$&$2.00$     &$5.9367e-06$&$2.00$ &$5.9367e-06$&$2.00$    &$6.2918e-06 $&$2.00 $   \\
$2^{13}$ &$1.7419e-06$&$-$    &$1.7419e-06$&$-$        &$1.7419e-06$&$-$    &$1.7419e-06$&$-$       &$1.8493e-06$&$-$      \\\hline
 $\epsilon$&\multicolumn{2}{c}{$2^{-25}$}
            &\multicolumn{2}{c}{$2^{-30}$}&\multicolumn{2}{c}{$2^{-35}$}
             &\multicolumn{2}{c}{$2^{-40}$}&\multicolumn{2}{c}{$2^{-45}$}\\\hline
\end{tabular}
\caption{Error $E_N$ and convergence rates Ord for approximate solution for example \ref{primjer2}.}
\end{table}
\end{example}

In the analysis of examples 1 and 2 from section \ref{sec5} and the corresponding result tables, we can observe the robustness of the constructed difference scheme, even for small values of the perturbation parameter $\epsilon$.
Note that the results presented in tables 1 and 2 already suggest $\epsilon$-uniform convergence of second order.

The presented method can be used in order to construct schemes of convergence order greater than two. In constructing such schemes, the corresponding analysis should not be more difficult that the analysis for our constructed difference scheme. In the case of constructing schemes for solving a two-dimensional singularly perturbed boundary value problem, if one does not take care that functions of two variables do not appear during the scheme construction, the analysis should not be substantially more difficult then for our constructed scheme. In such a case it would be enough to separate the expressions with the same variables and the analysis is reduced to the previously done one-dimensional analysis.

\begin{acknowledgements}
The authors are grateful to Nermin Oki\v ci\' c and Elvis Barakovi\' c for helpful advice.
Helena Zarin is supported by the Ministry of Education and Science of the Republic of Serbia under grant no. 174030.
\end{acknowledgements}

%%%%%%%%%%%%%%%%%%%%%%%%%%%%%%%%%%%%%


\begin{thebibliography}{28}
\bibitem{B1}{  Bakhvalov, N.S.:}{  Towards optimization of the methods for solving boundary value problems in presence of a boundary layers.} Zh.Vychisl. Mat i Mat. Fiz. {\bf 9}, 841--859  (in Russian) (1969)
%\bibitem{DA1}{ D'Annunzio, C.M.: }{  Numerical analysis of a singular perturbation problem with multiple solution.} Ph.D.Dissertation, University of Maryland at College Park (1986)
\bibitem{D4}{ Duvnjakovi\'c, E.: }{  A Class of Difference schemes for singular perturbation problem, } Proceedings of the 7th International Conference on Operational research, Croatian OR Society,  197--208 (1999).
\bibitem{H1}
{  Herceg, D.: } {  Uniform fourth order difference schemes for a singular perturbation problem,} Numer.Math. {\bf 56}, 675--693 (1990)
\bibitem{HH}{ Herceg, D., Herceg, Dj.: }{  On a Fourth-Order Finite Difference Methods for Nonlinear Two-Point Boundary Value Problems,}
               Novi Sad J.Math. {\bf 33}, 2,  173--180 (2003)
\bibitem{H2} {  Herceg, D., Miloradovi\' c, M.: }{  On numerical solution of semilinear singular perturbation problems by using the Hermite scheme on a new Bakhvalov-type mesh,} J.Math. Novi Sad {\bf 33} 1, 145--162 (2003)
\bibitem{S1}{  Herceg, D.,  Surla, K.: }{  Solving a Nonlocal Singularly Perturbed Problem by Splines in Tension,}
%             Univ. u Novom Sadu, Zb. Rad. Prirod.-Mat.Fak, {\bf 21} 2,  119--132 (1991)
\bibitem{HSR}{ Herceg, D.,  Surla, K.,  Rapaji\' c, S.: }{Cubic Spline Difference Scheme on a Mesh of Bakhvalov Type,}
               Novi Sad J. Math., {\bf 28} 3,  41--49 (1998)
\bibitem {Z1} Linss, T.: {Layer-Adapted Meshes for Reaction-Convection-Diffusion Problems}, Springer-Verlag, Berlin, Heidelberg (2010)
\bibitem{KN}{ Niijima,  K.: }{  A Uniformly Convergent Difference Scheme for a Selinear Singular Perturbation Problem,}
            Numer. Math. {\bf 43},  175--198 (1984)
\bibitem{OS}{  O'Riordan, E., Stynes, M.: }{ A uniformly accurate finite-element methods for a singularly perturbed one-dimensional reaction-diffusion problem,} Mathematics of Computation, {\bf 47} 176, 555--570 (1986).
\bibitem{OR}{  Ortega, J.M.,  Rheinboldt, W.C.: }{  Iterative solution of nonlinear equations in several variables,} SIAM, Philadelphia, USA (2000)
\bibitem{R}{  Roos,  H.-G.: }{Global uniformly convergent schemes for a singularly perturbed boundary-value problem using patched base spline-functions,} Journal of Computational and Applied Mathematics {\bf 29}, 69-77 (1990)
\bibitem{SH1}{   Shishkin, G.I. }{  Grid approximation of singularly perturbed parabolic equations with internal layers, } Sov.J.Numer.Anal.Math.Modelling {\bf 3} 393--407 (1988)
\bibitem{ORS1}{  Stynes, M.,  O'Riordan, E.: }{  $L^{1}$ and $L^{\infty}$ Uniform Convergence of a Difference Scheme for
             a Semilinear Singular Perturbation Problem,} Numer. Math. {\bf 50},  519--531 (1987)
\bibitem{SS}{   Sun, G.,  Stynes, M.: }{  A uniformly convergent method for a singularly perturbed semilinear reaction-diffusion problem with multiple solutions, } MATHEMATICS OF COMPUTATION, {\bf 65} 215, 1085--1109 (1996)
%      \bibitem{S3}{  Surla, K.: }{  On the Convergence of Some Finite Difference for a Singular Perturbation Problem, }
%           Review of Research, Faculty of Science-University of Novi Sad, {\bf 12}, 191--203 (1982)
%             \bibitem{S4}{  Surla, K.: }{  The Singularly Perturbed Spline Collocation Method for Boundary Value Problems with Mixed Boundary Conditions,}
%           Review of Research, Faculty of Science-University of Novi Sad, Mathematics Series {\bf 16} 2, 131--143 (1986)
%  \bibitem{S2}{  Surla, K. : }{  A Uniformly Convergent Spline Difference Scheme for a Self-Adjoint Singular Perturbation Problem, }
%             Univ. u Novom Sadu, Zb. Rad. Prirod.-Mat.Fak, {\bf 17} 2,  31--38 (1987)
%  \bibitem{S6}{  Surla, K.: }{  On Numerical Solving Singularly Perturbed Boundary Value Problems by Spline in Tension, }
%           Review of Research, Faculty of Science-University of Novi Sad, Mathematics Series {\bf 24} 2, 175--186 (1994)
%  \bibitem{S7}{  Surla, K.: }{  Some Difference Schemes Derived Via Finite Element Methods, }
%           Review of Research, Faculty of Science-University of Novi Sad, Mathematics Series {\bf 24} 1, 341--346 (1994)
%        \bibitem{S5}{  Surla, K., Jerkovi\' c, V.:}{  Collocation Spline Methods in Solving Boundary Value Problems, }
%           Review of Research, Faculty of Science-University of Novi Sad, Mathematics Series {\bf 23} 2, 345--350 (1993)
             \bibitem{S9}{   Surla, K., Uzelac, Z.: }{  A Uniformly Accurate Difference Scheme for Singular Perturbation Problem, }
           Indian J.pure appl. Math., {\bf 27} 10,  1005--1016 (1996)
%  \bibitem{S8}{   Surla K., Vukoslav\v cevi\' c, V.: }{  A Spline Difference Schemes for Boundary Value Problem with Small Parameter, }
%           Review of Research, Faculty of Science-University of Novi Sad, Mathematics Series {\bf 25} 2, 159--168 (1995)
\bibitem{V1}{  Vulanovi\'c, R.: }{  On a numerical solution of a type of singularly perturbed boundary value problem by using a special discretization mesh, } Review of Research Faculty of Science-University of Novi Sad, {\bf 13} (1983)
\bibitem{V2} {  Vulanovi\'c, R. }{  On numerical solution of semilinear singular perturbation problems by using the Hermite scheme, } Review of Research Faculty of Science-University of Novi Sad, {\bf 23} 2, 363--379 (1993)
  \bibitem{V3}{  Vulanovi\' c, R.: }{  An almost sixth-order finite-difference method for semilinear singular perturbation problems, } CMAM {\bf 4}, 368--383 (2004)
\bibitem{SU1}{  Uzelac, Z., Surla, K.:} {  A uniformly accurate collocation method for a singularly perturbed problem, } Novi Sad Journal of Mathematics, {\bf 33} 1,  133--143 (2003)











\end{thebibliography}
\end{document}